\documentclass[a4paper,12pt]{article}
\usepackage{amsfonts}
\usepackage{mathrsfs}
\usepackage{algorithm, algpseudocode}
\usepackage{amsmath,amssymb,amsthm,latexsym,amsfonts}
\usepackage{pstricks}
\usepackage{enumerate}

\usepackage{graphicx}
\usepackage{color}
\usepackage{ifpdf}

\usepackage{caption}

\begin{document}

\newtheorem{lem}{Lemma}
\newtheorem{thm}{Theorem}
\newtheorem{cor}{Corollary}
\newtheorem{exa}{Example}
\newtheorem{con}{Conjecture}
\newtheorem{rem}{Remark}
\newtheorem{obs}{Observation}
\newtheorem{definition}{Definition}
\newtheorem{pro}{Proposition}
\theoremstyle{plain}
\newcommand{\D}{\displaystyle}
\newcommand{\DF}[2]{\D\frac{#1}{#2}}

\renewcommand{\figurename}{{\bf Fig}}
\captionsetup{labelfont=bf}

\title{\bf \Large Rainbow vertex-connection and forbidden
subgraphs\footnote{Supported by NSFC No.11371205 and 11531011, and PCSIRT.}}

\author{{\small Wenjing Li, Xueliang Li, Jingshu Zhang}\\
      {\small Center for Combinatorics and LPMC}\\
      {\small Nankai University, Tianjin 300071, China}\\
       {\small liwenjing610@mail.nankai.edu.cn; lxl@nankai.edu.cn;
       jszhang@mail.nankai.edu.cn}
       }
\date{}

\maketitle
\begin{abstract}
A path in a vertex-colored graph is called \emph{vertex-rainbow} if its internal vertices
have pairwise distinct colors. A graph $G$ is \emph{rainbow vertex-connected}
if for any two distinct vertices of $G$, there is a vertex-rainbow path connecting them.
For a connected graph $G$, the \emph{rainbow vertex-connection number} of $G$, denoted by $rvc(G)$,
is defined as the minimum number of colors that are required to make $G$ rainbow vertex-connected.
In this paper, we find all the families $\mathcal{F}$ of connected graphs with $|\mathcal{F}|\in\{1,2\}$,
for which there is a constant $k_\mathcal{F}$ such that,
for every connected $\mathcal{F}$-free graph $G$, $rvc(G)\leq diam(G)+k_\mathcal{F}$,
where $diam(G)$ is the diameter of $G$.
\\[2mm]


\end{abstract}

\section{Introduction}

All graphs considered in this paper are simple, finite, undirected and connected.
We follow the terminology and notation of Bondy and Murty in \cite{Bondy} for those not
defined here.

Let $G$ be a nontrivial connected graph with an \emph{edge-coloring c}
$:E(G)\rightarrow \{0,1,\dots,t\}$, $t\in \mathbb{N}$, where adjacent edges may
be colored with the same color. A path in $G$ is called \emph{a rainbow path} if no two edges
of the path are colored with the same color. The graph $G$ is called \emph{rainbow connected}
if for any two distinct vertices of $G$, there is a rainbow path connecting them.
For a connected graph $G$, the \emph{rainbow connection number} of $G$, denoted by $rc(G)$, is defined as the
minimum number of colors that are needed to make $G$ rainbow connected.
Observe that if $G$ has $n$ vertices, then $diam(G)\leq rc(G)\leq n-1$.
And, it is easy to verify that $rc(G)=1$ if and only if $G$ is a complete graph,
and $rc(G)=n-1$ if and only if $G$ is a tree. The concept of rainbow connection of graphs
was first introduced by Chartrand et al. in \cite{Char1}, and has been well-studied since then.
For further details, we refer the reader to a book \cite{Li}.

Let $G$ be a nontrivial connected graph with a \emph{vertex-coloring c}
$:V(G)\rightarrow \{0,1,\dots,t\}$, $t\in \mathbb{N}$,
where adjacent vertices may be colored with the same color.
A path of $G$ is called {\it vertex-rainbow} if any two internal vertices of the
path have distinct colors. The graph $G$ is \emph{rainbow vertex-connected}
if any two vertices of $G$ are connected
by a vertex-rainbow path. For a connected graph $G$, the \emph{rainbow vertex-connection number}
of $G$, denoted by $rvc(G)$, is the minimum number of colors used in a vertex-coloring of $G$
to make $G$ rainbow vertex-connected.
The concept of rainbow vertex-connection of graphs
was proposed by Krivelevich and Yuster in \cite{Yuster}.
They showed that if $G$ is a connected graph with
$n$ vertices and minimum degree $\delta$,
then $rvc(G)\leq 11n/\delta$. In \cite{X. Li2}, Li and Shi improved this bound.
In \cite{X. Li3}, it was shown that computing the rainbow vertex-connection
number of a graph is NP-hard.

For the rainbow vertex-connection number of graphs, the following observations
are immediate.

\begin{pro}\label{pro1}
Let $G$ be a connected graph with $n$ vertices. Then

$(i)\ diam(G)-1\leq rvc(G)\leq n-2;$

$(ii)\ rvc(G)=diam(G)-1$ if $diam(G)=1$ or $2$, with the assumption
that complete graphs have rainbow vertex-connection number 0.
\end{pro}

Note that the difference $rvc(G)-diam(G)$ can be arbitrarily large. In fact,
If $G$ is a subdivision of a star $K_{1,n}$, then we have $rvc(G)-diam(G)=
(n+1)-4=n-3$, since every internal vertex requires a single color.

In \cite{X. Li4}, Li and Liu studied the rainbow vertex-connection number
for any 2-connected graph, and determined the precise value of the
rainbow vertex-connection number of the cycle $C_n\ (n\geq 3)$.

\begin{thm}\cite{X. Li4}\label{thm1}
Let $C_n$ be a cycle of order $n\ (n\geq 3)$. Then,
\begin{displaymath}
rvc(C_n) = \left\{\begin{array}{ll}
0  & \text{if $n=3$;}\\
1  & \text{if $n=4,5$;}\\
3  & \text{if $n=9$;}\\
\lceil\frac{n}{2}\rceil -1 & \text{if $n=6,7,8,10,11,12,13$,or $15$;}\\
\lceil\frac{n}{2}\rceil  & \text{if $n\geq 16$ or $n=14$.}
\end{array}\right.
\end{displaymath}
\end{thm}

Let $\mathcal{F}$ be a family of connected graphs. We say that a graph
$G$ is \emph{$\mathcal{F}$-free} if $G$ does not contain any induced subgraph
isomorphic to a graph from $\mathcal{F}$. Specifically, for $\mathcal{F}=\{X\}$
we say that $G$ is \emph{X-free}, and for $\mathcal{F}=\{X,Y\}$
we say that $G$ is \emph{(X,Y)-free}. The members of $\mathcal{F}$ will be referred to
in this context as \emph{forbidden induced subgraphs},
and for $|\mathcal{F}|=2$ we also say that $\mathcal{F}$ is a \emph{forbidden pair}.

In \cite{P.1}, Holub et al. considered the question: For which families $\mathcal{F}$
of connected graphs, a connected $\mathcal{F}$-free graph $G$ satisfies $rc(G)\leq diam(G)+k_\mathcal{F}$,
where $k_\mathcal{F}$ is a constant (depending on $\mathcal{F}$),
and they gave a complete answer for $|\mathcal{F}|\in\{1,2\}$ in the following two
results (where $N$ denotes the \emph{net}, a graph obtained by attaching a pendant edge
to each vertex of a triangle).

\begin{thm}\cite{P.1}\label{thm2}
Let $X$ be a connected graph. Then
there is a constant $k_\mathcal{F}$ such that every connected $X$-free graph $G$
satisfies $rc(G)\leq diam(G)+k_X$,
if and only if $X=P_3$.
\end{thm}

\begin{thm}\cite{P.1}\label{thm3}
Let $X, Y$ be connected graphs such that $X,Y \neq P_3$. Then
there is a constant $k_{XY}$ such that every connected $(X,Y)$-free graph $G$
satisfies $rc(G)\leq diam(G)+k_{XY}$,
if and only if (up to symmetry) either $X=K_{1,r} \ (r\geq4)$ and $Y=P_4$,
or $X=K_{1,3}$ and $Y$ is an induced subgraph of $N$.
\end{thm}

Naturally, we wonder an analogous question concerning the rainbow vertex-connection
number of graphs. In this paper, we will consider the following question.

\emph{For which families $\mathcal{F}$
of connected graphs, there is a constant $k_\mathcal{F}$ such that
a connected
graph $G$ being $\mathcal{F}$-free implies
$rvc(G)\leq diam(G)+k_\mathcal{F}$?}

We give a complete answer for $|\mathcal{F}|=1$ in Section 3,
and for $|\mathcal{F}|=2$ in Section 4.

\section{Preliminaries}

In this section, we introduce some further notations and facts that will be
needed for the proofs of our main results.

If $G$ is a graph and $A\subset V(G)$,
then $G[A]$ denotes the subgraph of $G$
induced by the vertex set $A$, and $G-A$
the graph $G[V(G)\backslash\ A]$.
An edge is called a \emph{pendant edge} if one of its end vertices has
degree one. The \emph{subdivision} of a graph $G$ is the graph obtained
from $G$ by adding a vertex of degree 2 to each edge of $G$.
For $x,y\in V(G)$, a path in $G$ from $x$ to $y$ will be referred to as an
\emph{$(x,y)$-path}, and, whenever necessary, it will be considered as oriented
from $x$ to $y$.
For a subpath of a path $P$ with origin $u$ and terminus $v$
(also referred to as a \emph{$(u,v)$-arc} of $P$),
we will use the notation $uPv$.
If $w$ is a vertex of a path with a fixed orientation,
then $w^-$ and $w^+$ denote the predecessor and successor of $w$, respectively.

For graphs $X$ and $G$, we write $X\subset G$ if $X$ is a subgraph of
$G$, $X\overset{\scriptscriptstyle{\text{IND}}}{\subset}G$
if $X$ is an induced subgraph of
$G$, and $X\simeq G$ if $X$ is isomorphic to $G$.
For two vertices $x,y\in V(G)$, we use $dist_G(x,y)$ to denote the distance
between $x$ and $y$ in $G$. The diameter of $G$ is defined as the maximum
of $dist_G(x,y)$ among all pairs of vertices $x,y$ of $G$, and will be denoted by
$diam(G)$. A shortest path joining two vertices at distance $diam(G)$
will be referred to as a \emph{diameter path}.
The \emph{distance between a vertex $u\in V(G)$ and a set $S\subset V(G)$} is defined
as $dist_G(u,S):=min_{v\in S}dist_G(u,v)$.
A set $D\subset G$ is called \emph{dominating} if every vertex in
$V(G)\setminus D$ has a neighbor in $D$.
In addition, if $G[D]$ is connected, then we call $D$
a \emph{connected dominating set}.
Throughout this paper, $\mathbb{N}$ denotes the set of all
positive integers.

For a set $S\subset V(G)$ and $k\in \mathbb{N}$, the \emph{kth-neighborhood} of $S$ is the set $N_G^k(S)$ of all vertices of $G$ at distance $k$ from $S$. In the special case $k=1$, we simply write $N_G(S)$ for $N_G^1(S)$, and if $|S|=1$ with $x\in S$, we write $N_G(x)$ for $N_G(\{x\})$. For a set $M\subset V(G)$, we set $N_M^k(S)=N_G^k(S)\cap M$ and $N_M^k(x)=N_G^k(x)\cap M$, and for a subgraph $P\subset G$, we write $N_P(x)$ for $N_{V(P)}(x)$. Finally, we will use $P_k$ to denote the path on $k$ vertices.

We end up this section with an important result that will be used in our proofs.

\begin{thm}\cite{G.}\label{thm4}
Let $G$ be a connected $P_5$-free graph. Then $G$ has a dominating clique
or a dominating $P_3$.
\end{thm}

\section{Families with one forbidden subgraph}

In this section, we characterize all connected graphs $X$ such that every connected
$X$-free graph $G$ satisfies $rvc(G)\leq diam(G)+k_X$, where $k_X$ is a constant.

\begin{thm}\label{thm5}
Let $X$ be a connected graph. Then there is a constant $k_X$
such that every connected $X$-free graph $G$
satisfies $rvc(G)\leq diam(G)+k_X$,
if and only if $X = P_3$ or $P_4$.
\end{thm}

\begin{proof}
We have $diam(G)\leq 2$ since $G$ is $P_4$-free. Then it follows Proposition
\ref{pro1}  that $rvc(G)=diam(G)-1\leq 1$.

Conversely, let $t\geq k_X+5$, and
$G_1^t$ be the subdivision of $K_{1,t}$,
and let $G_2^t$ denote the graph obtained by attaching a pendant edge
to each vertex of the complete graph $K_t$ (see Fig.1).
Since $rvc(G_1^t) = t$ but $diam(G_1^t) = 4$,
$X$ is an induced subgraph of $G_1^t$.
Clearly, $rvc(G_2^t) = t$ but $diam(G_2^t) = 3$,
and $G_2^t$ is $K_{1,3}$-free and $P_5$-free.
Hence, $X$ is an induced subgraph of $P_4$.

The proof is thus complete.
\end{proof}

\begin{figure}[h,t,b,p]
\begin{center}
\scalebox{1.2}[1.2]{\includegraphics{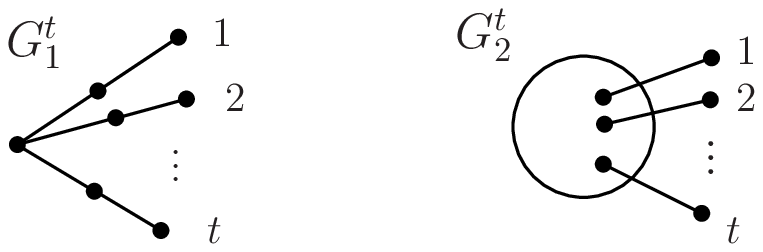}}\\[12pt]
Figure~1: The graphs $G_1^t$ and $G_2^t$.
\end{center}
\end{figure}

\section{Families with a pair of forbidden subgraphs}

For $i,j,k\in \mathbb{N}$, let $S_{i,j,k}$ denote
the graph obtained by identifying one endvertex
from each of three vertex-disjoint paths of length $i,j,k$, and
$N_{i,j,k}$ denote the graph obtained by identifying each vertex of a triangle
with an endvertex of one of three vertex-disjoint paths
of length $i,j,k$ (see Fig.2). In this context, we will also write $K_t^h$ for the graph
$G_2^t$ introduced in the proof of Theorem \ref{thm5}.

\begin{figure}[h,t,b,p]
\begin{center}
\scalebox{1}[1]{\includegraphics{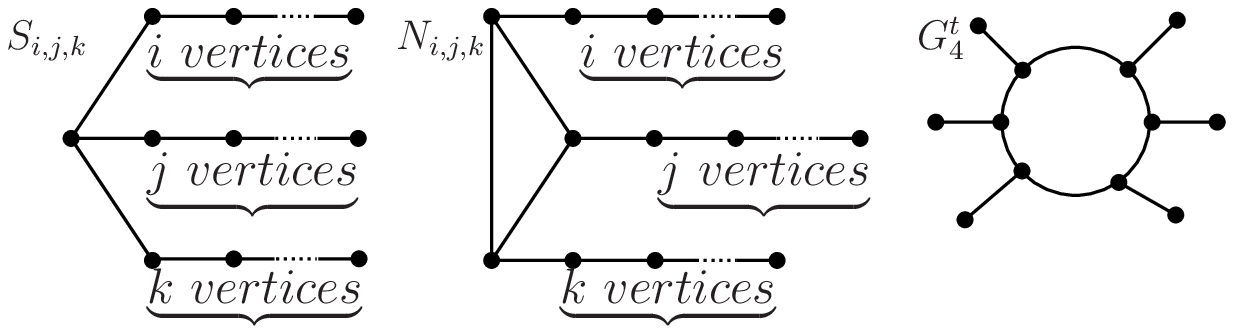}}\\[12pt]
Figure~2: The graphs $S_{i,j,k}$, $N_{i,j,k}$ and $G_4^t$.
\end{center}
\end{figure}

The following statement, which is the main result of this section,
characterizes all forbidden pairs
$X,Y$ for which there is a constant $k_{XY}$
such that $G$ being $(X,Y)$-free implies $rvc(G)\leq diam(G)+k_{XY}$.
By virtue of Theorem \ref{thm5}, we exclude the case that one of $X,Y$
is an induced subgraph of $P_4$.

\begin{thm}\label{thm6}
Let $X,Y\neq P_3$ or $P_4$ be a pair of connected graphs.
Then there is a constant $k_{XY}$
such that every connected $(X,Y)$-free graph $G$
satisfies $rvc(G)\leq diam(G)+k_{XY}$,
if and only if (up to symmetry)
$X = P_5$ and $Y\overset{\scriptscriptstyle{\text{IND}}}{\subset} K_t^h \ (t\geq 4)$, or
$X\overset{\scriptscriptstyle{\text{IND}}}{\subset} S_{1,2,2}$
and $Y \overset{\scriptscriptstyle{\text{IND}}}{\subset} N$.
\end{thm}

The proof of Theorem \ref{thm6} will be divided into three
separate results: we prove the necessity in Proposition \ref{pro2},
and Theorems \ref{thm7} and \ref{thm8} will establish the sufficiency of the forbidden
pairs given in Theorem \ref{thm6}.

\begin{pro}\label{pro2}
Let $X,Y\neq P_3$ or $P_4$ be a pair of connected graphs
for which there is a constant $k_{XY}$
such that every connected $(X,Y)$-free graph $G$
satisfies $rvc(G)\leq diam(G)+k_{XY}$.
Then, (up to symmetry)
$X = P_5$ and $Y\overset{\scriptscriptstyle{\text{IND}}}{\subset} K_t^h\ (t\geq 4)$, or
$X\overset{\scriptscriptstyle{\text{IND}}}{\subset} S_{1,2,2}$
and $Y \overset{\scriptscriptstyle{\text{IND}}}{\subset} N$.
\end{pro}

\begin{proof}
Let $t\geq 2k_{XY}+5$, and let (see Fig.2):

$\bullet\ G_3^t = N_{t-1,t-1,t-1};$

$\bullet\ G_4^t$ be the graph obtained by attaching a pendant edge
to each vertex of a cycle $C_t$.

We will also use the graphs $G_1^t$ and $G_2^t(=K_t^h)$ shown in Fig.1.

For the graphs $G_1^t$ and $G_2^t$, we have $diam(G_1^t) = 4$
but $rvc(G_1^t) = t$, and $diam(G_2^t) = 3$
but $rvc(G_2^t) = t$, respectively.
For the graph $G_3^t$, we observe that $diam(G_3^t) = 2t-1$
while $rvc(G_3^t) \geq 3(t-1) = \frac{3}{2}(diam(G_3^t)-1)$,
since all internal vertices must have mutually
distinct colors. Analogously, for the graph $G_4^t$, we have
$diam(G_4^t) = \lfloor\frac{t}{2}\rfloor+2$,
but $rvc(G_4^t) = t \geq 2(diam(G_4^t)-2)$.
Thus, each of the graphs $G_1^t$, $G_2^t$, $G_3^t$
and $G_4^t$ must contain an induced subgraph isomorphic to one of the graphs $X,Y$.

Consider the graph $G_1^t$. Up to symmetry, we have that $X$ is an induced subgraph of
$G_1^t$ excluding $P_3$ and $P_4$.
Now we consider the graph $G_2^t$.
Obviously, $G_2^t$ is $X$-free since $G_2^t$ is $K_{1,3}$-free.
Hence, $G_2^t$ contains $Y$,
implying $Y \overset{\scriptscriptstyle{\text{IND}}}{\subset} K_t^h$ for some $t \geq 3$ (for $t \leq 2$ we get $Y\overset{\scriptscriptstyle{\text{IND}}}{\subset} P_4$,
which is excluded by the assumptions).

Now we consider the graph $G_3^t$. There are two possibilities:

$(i)\ Y\overset{\scriptscriptstyle{\text{IND}}}{\subset}G_3^t$.
Then $Y \overset{\scriptscriptstyle{\text{IND}}}{\subset} N$. Now we consider the graph $G_4^t$.
$G_4^t$ is $N$-free, so we get $X \overset{\scriptscriptstyle{\text{IND}}}{\subset} S_{1,2,2}$.

$(ii)\ X\overset{\scriptscriptstyle{\text{IND}}}{\subset}G_3^t$.
Then $X = P_5$. As the case $X = P_5$ and $Y = N$
is already covered by case $(i)$, we have that
$X = P_5$ and $Y \overset{\scriptscriptstyle{\text{IND}}}{\subset} K_t^h,\ t\geq 4$.

This completes the proof.
\end{proof}

It is easy to observe that if $X\overset{\scriptscriptstyle{\text{IND}}}{\subset}
X'$, then every $(X,Y)$-free graph is also $(X',Y)$-free.
Thus, when proving the sufficiency of Theorem
\ref{thm6},
we will be always interested in \emph{maximal pairs} of forbidden subgraphs,
i.e., pairs $X,Y$ such that, if replacing one of $X,Y$, say $X$, with
a graph $X'\neq X$ such that $X\overset{\scriptscriptstyle{\text{IND}}}{\subset}
X'$, then the statement under consideration is not true for $(X',Y)$-free graphs.

\begin{thm}\label{thm7}
Let $G$ be a connected $(P_5,K_t^h)$-free graph for some $t\geq 4$.
Then, $rvc(G)\leq diam(G)+t$.
\end{thm}

\begin{proof}
From Theorem \ref{thm1}, we have that $G$ has a dominating clique or a dominating
$P_3$.

{\bf Case 1:} $G$ has a dominating $P_3$.

We color the vertices of $P_3$
with colors $1,2,3$ and color the remaining vertices arbitrarily (e.g.,
all of them with color 1). One can easily check that this vertex-coloring
can make $G$ rainbow vertex-connected. So, in this case, $rvc(G) \leq 3\leq diam(G)+t$.

{\bf Case 2:} $G$ has a dominating clique, denoted by $K_p$.

Set $W=V(G)\backslash V(K_p)$, $H=G\backslash E(K_p)$.
Let $A$ be an independent set in $G[W]$ and $B\subset V(K_p)$
such that $H[A\cup B]=\ell K_2$ (that is, a matching of order $\ell$)
and $\ell$ is maximal. Then $\ell <t$,
for otherwise, $G[A\cup B]$ contains an induced $K_t^h$.
Moreover, for $x\in W\backslash A$, $N_{A\cup B}(x)\neq \emptyset$,
since $\ell$ is maximal.
Now we demonstrate a rainbow vertex-coloring of $G$.
Use colors $1,2, \dots, \ell$ to color each vertex in $B$,
color the vertices of $A$ with color $\ell+1$, the vertices of
$V(K_p)\backslash B$ with color $\ell+2$,
and color the remaining vertices arbitrarily.
Thus, pairs in $(A\cup V(K_p))\times (A\cup V(K_p))$ and
$(A\cup V(K_p))\times (W\backslash A)$ are rainbow vertex-connected.
As for $x_1,x_2\in W\backslash A$, let $y_1\in N_{A\cup B}(x_1)$,
$y_2\in N_{K_p}(x_2)$. Then, there is a rainbow $(x_1,x_2)$-path containing $y_1$ and $y_2$.
So, $rvc(G)\leq \ell+2\leq t+1\leq diam(G)+t$.

The proof is complete.  \end{proof}

Now let $G$ be an $(S_{1,2,2},N)$-free graph, let $x, y\in V(G)$,
and let $P:x = v_0, v_1, \dots, v_k = y \ (k\geq 3)$
be a shortest $(x,y)$-path in $G$.
Let $z\in V(G)\backslash V(P)$. If $|N_P(z)|\geq 2$
and $\{v_i,v_j\}\subset N_P(z)$, then
$|i-j| \leq 2$ and $|N_P(z)|\leq 3,$ since $P$ is a shortest path.
Moreover, the following facts are easily observed.

$\bullet$ If $|N_P(z)|=1$, then, since $G$ is $S_{1,2,2}$-free,
$z$ is adjacent to $x, v_1, v_{k-1}$ or to $y$.

$\bullet$ If $|N_P(z)|=3$, then the vertices of $N_P(z)$ must be
consecutive on $P$, since $P$ is a shortest path.

This motivates the following notations:

$\bullet\ A_i:=\{ z\in V(G)\backslash V(P)| N_P(z) = \{v_i\}\}$
for $i = 0,1, k-1, k$;

$\bullet\ L_i:=\{ z\in V(G)\backslash V(P)| N_P(z) = \{v_{i-1}, v_{i+1}\}\}$
for $1\leq i \leq k-1$;

$\bullet\ M_i:=\{ z\in V(G)\backslash V(P)| N_P(z) = \{v_{i-1}, v_i\}\}$
for $1\leq i \leq k$;

$\bullet\ N_i:=\{ z\in V(G)\backslash V(P)| N_P(z) = \{v_{i-1}, v_i, v_{i+1}\}\}$
for $1\leq i \leq k-1$.

We further set $S=V(P)\cup N_G(P)$ and $R=V(G)\backslash S$.

\begin{lem}\label{lem1}
Let $G$ be an $(S_{1,2,2},N)$-free graph, let $x,y\in V(G)$
be such that $dist_G(x,y)\geq 4$
and let $P: x = v_0, v_1, \dots, v_k = y$,
be a shortest $(x,y)$-path in $G$. Then

$(i)\ N_G(M_i) \subset S, \ i=2,\dots, k-1$;

$(ii)\ N_G(N_i) \subset S, \ i=2,\dots, k-2$;

$(iii)\ N_G(L_i) \subset S, \ i=1,\dots, k-1$;

$(iv)\ N_P(R) = \emptyset$;

$(v)\ N_S(R)\subset A_0\cup M_1\cup N_1\cup N_{k-1}\cup M_k\cup A_k$.
\end{lem}

\begin{proof}
If $zv\in E(G)$ for some $z\in R$ and $v\in M_i,\ 2\leq i\leq k-1$,
then we have $G[\{v_{i-2},v_{i-1},v_i,$ $v_{i+1},v,z\}]\simeq N$, a contradiction.
Hence, $(i)$ follows.
To show $(ii)$, we observe that if $zv\in E(G)$ for some $z\in R$ and
$v\in N_i,\ 2\leq i\leq k-2$,
then we have $G[\{v_{i-2},v_{i-1},v_{i+1},v_{i+2},v,z\}]\simeq S_{1,2,2}$, a contradiction.
Similarly, if $zv\in E(G)$ for some $z\in R$ and
$v\in L_i,\ 1\leq i\leq k-1$,
then, for $i=1$ we have
$G[\{v_1,v_2,v_3,v_4,v,z\}]\simeq S_{1,2,2}$,
for $ 2\leq i\leq k-2$ we have
$G[\{ z,v,v_{i-1},v_{i-2},v_{i+1},v_{i+2}\}]\simeq S_{1,2,2}$,
and for $i=k-1$,
$ G[\{ v_{k-1},v_{k-2},v_{k-3},v_{k-4},v,z\}]\simeq S_{1,2,2}$,
a contradiction.
Part $(iv)$ follows immediately from the definition of $R$,
and by $(i)$ through $(iii)$, we have
$N_S(R)\subset A_0\cup A_1\cup M_1\cup N_1\cup N_{k-1}\cup M_k\cup A_{k-1}\cup A_k$.
But if $zv\in E(G)$ for some $z\in R$ and $v\in A_1$,
then $G[\{v_0,v_1,v_2,v_3,v,z\}]\simeq S_{1,2,2}$, a contradiction.
Similarly, we have $N_{A_{k-1}}(R) = \emptyset$, implying $(v)$.

The proof is complete.
\end{proof}

\begin{thm}\label{thm8}
Let $G$ be a connected $(S_{1,2,2},N)$-free graph.
Then, $rvc(G)\leq diam(G)+11$.
\end{thm}

\begin{proof}
Let $G$ be a connected $(S_{1,2,2},N)$-free graph.
If $diam(G)\leq 2$, then $rvc(G)=diam(G)-1$.
Thus, for the rest of the proof we suppose that $diam(G) = d \geq 3$.
Let $v_0,v_d\in V(G)$ be such that $dist_G(v_0,v_d)=d$,
let $P:v_0v_1v_2\dots v_d$ be a diameter path in $G$,
and let $A_i, L_i, M_i, N_i, S, R$ be defined as above.

We distinguish three cases according to the value of $d$.

{\bf Case 1:} $d=3.$

First, we partition $V(G)$ into four parts
$P,N_G^1(P), N_G^2(P)$
and $N_G^3(P)$ according to the distance
from $P$. Then, for the vertices in $N_G^1(P)$,
we can partition them into three parts
$X_1=A_0\cup M_1\cup L_1\cup N_1$,
$X_2=A_3\cup M_3\cup L_2\cup N_2$
and $X_3=A_1\cup M_2 \cup A_2$.
We must point out that $X_1\cap X_2=\emptyset$
and $N_R(X_3)=\emptyset$,
whose proof is similar to that of Lemma \ref{lem1}.
Then, we denote $Y_i$ the set of vertices
in $N_G^2(P)$ such that for each $v\in Y_i$,
$N_{N(P)}(v)\subset X_i,i=1,2$,
and $Y_3=N_G^2(P)\backslash (Y_1\cup Y_2)$.
And with a similar reason as above,
$N_{N_G^3(P)}(Y_3)=\emptyset$.
So, analogously we can partition $N_G^3(P)$
into three parts $Z_1,Z_2$ and $Z_3$.
Since for a vertex $z\in Z_1$, $dist_G(z,v_3)\geq 4$,
it follows that $Z_1=\emptyset$, a contradiction.
Symmetrically, we have $Z_2=\emptyset$.

Now, we demonstrate a rainbow vertex-coloring of $G$
that uses at most $14$ colors.
Color the vertices of $P$ with colors $0,1,2,3$
and color the vertices in $A_0, M_1, L_1, N_1, N_2,$ $L_2,M_3,A_3, Y_1$ and $Y_2$ with colors $4,5,\dots,13$, respectively. Then, color the remaining vertices arbitrarily (e.g., all
of them with color 0). We can show that this vertex-coloring can make $G$ rainbow vertex-connected.
We only need to verify that for a pair $(x,y)\in (Y_1\times Y_1)\cup (Y_2\times Y_2)$, there exists a rainbow path connecting them. Without loss of generality,
we suppose $(x,y) \in Y_1\times Y_1$.
If $dist_G(x,y)\leq 2$, then there is nothing left to do.
Next we consider $dist_G(x,y)\geq 3$.
Let $x'$ be an arbitrary neighbor of $x$ in $X_1$,
and $y'$ an arbitrary neighbor of $y$ in $X_1$.
We claim that $x'$ and $y'$ cannot have the same color.
Otherwise, we suppose that $x'$ and $y'$ are colored with the same color,
i.e., they are in the same vertex-class of $X_1$,
and let $i=max\{j:v_j\in N_P(x')\cap N_P(y')\}$.
Then, we have $G[\{v_i,v_{i+1},x',y',x,y\}]\simeq S_{1,2,2}$
if $x'y'\notin E(G)$, or
$G[\{v_i,v_{i+1},x',y',x,y\}]\simeq N$ if $x'y'\in E(G)$,
respectively. So, the colors of $x'$ and $y'$
must be different. Then, the $(x,y)$-path $P_1:xx'v_0y'y$ is vertex-rainbow.
Hence, we have $rvc(G)\leq diam(G)+11$.

{\bf Case 2:} $d= 4.$

Similarly, with the partition and the vertex-coloring
of Case 1, we have $rvc(G)\leq 15=diam(G)+11$.

{\bf Case 3:} $d\geq 5.$

Set $B_c= (\cup_{i=2}^{d-2}N_i)\cup(\cup_{i=2}^{d-1}M_i)
\cup(\cup_{i=1}^{d-1}L_i)\cup A_1\cup A_{d-1}\cup\{v_1,v_2,\dots, v_{d-1}\}$,
$X= A_0\cup M_1\cup N_1\cup N_{d-1}\cup M_d\cup A_d$,
$X_1=A_0\cup M_1\cup N_1$,
and $X_2=N_{d-1}\cup M_d\cup A_d$.
By virtue of Lemma \ref{lem1}, we have $N_G(B_c)\subset S$.

{\bf Subcase 3.1:} $B_c$ is a cut-set of $G$.

We claim that $S\cup N_G(S)=V(G)$. Suppose, to the contrary, that $z\in R$
is at distance 2 from $S$. Then, by Lemma \ref{lem1} and the assumption
of Case 1, as well as the symmetry, we can assume that $N_S^2(z)\subset X_1$.
Let $Q$ be a shortest $(z,v_d)$-path, let $w$
be the first vertex of $Q$ in $B_c$ (it exists
by the assumption of Subcase 3.1),
and let $w^-$ be the predecessor of $w$ on $Q$.
By Lemma \ref{lem1}, $dist(w^-,P)=1$, implying $w^-\in X_1$.
Then, $dist_G(w^-,v_d)\geq d-1$;
otherwise, the path $v_0w^-Qv_d$ is a $(v_0,v_d)$-path
shorter than $P$. Since $dist_G(z,w^-)\geq 2$,
we have $dist_G(z,v_d)\geq d+1$, contradicting $diam(G) = d$.
Hence, we have $S\cup N_G(S)=V(G)$. Moreover, with a similar
argument to that of Case 1, we have that for $x,y\in R$ with distance at least 3,
their neighbors $x'$ and $y'$ cannot be in the same vertex-class of $X$.

Now we demonstrate a rainbow vertex-coloring of $G$
that uses at most $d+7$ colors.
Color the vertices of $P$ with colors $0,1,\dots, d$
and color the vertices in $A_0, M_1, N_1, N_{d-1},$ $M_d$
and $A_d$ with colors $d+1,d+2,\dots, d+6$, respectively.
Then, color the remaining vertices arbitrarily (e.g., all
of them with color 0). We can show that this vertex-coloring
can make $G$ rainbow vertex-connected.
For any pair of vertices in $S\times (S\cup R)$,
we can easily find a rainbow path connecting them.
For a pair $(x,y) \in R\times R$,
if $dist_G(x,y)\leq 2$, then there is nothing left to do.
Next we consider $dist_G(x,y)\geq 3$. From above,
we know that their neighbors $x'$ and $y'$ in $X$ are colored differently.
So, the $(x,y)$-path containing $x'$ and $y'$ is rainbow.

Consequently, we have $rvc(G)\leq diam(G)+7$.

{\bf Subcase 3.2:} $B_c$ is not a cut-set of $G$.

Set $H=G-B_c$. Let $P':v_dv_{d+1}\dots v_{d+\ell-1}v_{d+\ell}=v_0$
be a shortest $(v_d,v_0)$-path in $H$.
Since $P$ is a diameter path, $\ell\geq d\geq 5.$
If $v_{d+1}$ is adjacent to $v_{d-2}$,
then $G[\{v_d,v_{d+1},v_{d-2},v_{d-3},v_{d+2},$ $v_{d+3}\}]\simeq S_{1,2,2}$,
a contradiction. So, $v_{d+1}\in A_d\cup M_d$.
Similarly, we have $v_{d+\ell-1}\in A_0\cup M_1$.

Set $P^d:v_{d-1}v_dv_{d+1}$ if $v_{d-1}v_{d+1}\notin E(G)$,
or $P^d:v_{d-1}v_{d+1}$ if $v_{d-1}v_{d+1}\in E(G)$, respectively.
Similarly, set $P^0:v_{d+\ell-1}v_0v_1$ if $v_{d+\ell-1}v_1\notin E(G)$,
or $P^d:v_{d+\ell-1}v_1$ if $v_{d+\ell-1}v_1\in E(G)$, respectively.
Finally, set $C:v_1Pv_{d-1}P^dv_{d+1}P'v_{d+\ell-1}P^0v_1$.
Then, $C$ is a cycle of length at least $2d-2$.

{\bf Claim 1 \cite{P.1}:} The cycle $C$ is chordless.

\begin{proof}
Suppose, to the contrary, that $v_iv_j\in E(G)$ is a chord in $C$.
Since both $P$ and $P'$ are chordless, we can choose the notation
such that $1\leq i\leq d-1$ and $d+1\leq j\leq d+\ell-1$.
Since $v_j\in V(P')$, we have $v_j\notin B_c$ by the definition of $P'$,
implying $i=d-1$ and $v_j\in M_d$,
or, symmetrically, $i=1$ and $v_j\in M_1$.
This implies that in the first case, $v_j=v_{d+1}$;
in the second case, $v_j=v_{d+\ell-1}$; and
in both cases, $v_iv_j\in E(C)$ by the definition of $C$.
Thus, $C$ is chordless.
\end{proof}

{\bf Claim 2:} $\ell\leq d+2$.

\begin{proof}
Assume that $\ell\geq d+3$, and let $Q$ be a shortest $(v_0,v_{d+2})$-path
in $G$. Then, $|E(Q)|\leq d$ (since $diam(G)=d$).
And, since $\ell\geq d+3$ and $P'$ is shortest in $H=G-B_c$,
we have $dist_H(v_0,v_{d+2})\geq d+1.$
So, $Q$ must contain a vertex from $B_c$.
Let $w$ be the last vertex of $Q$ in $B_c$,
and let $w^-$ and $w^+$ be its predecessor and successor on $Q$, respectively
(they exist since $v_{d+2}\notin B_c$ by the definition of $P'$).
By Lemma \ref{lem1}, $w^+$ is at distance at most 1 from $P$.
Since clearly $w^+\notin \{v_0,v_d\}$,
either $w^+v_0\in E(G)$ or $w^+v_d\in E(G)$.
If $w^+v_0\in E(G)$, then $v_0w^+Qv_{d+2}$ is a $(v_0,v_{d+2})$-path
shorter than $Q$, a contradiction. Thus, $w^+v_d\in E(G)$.
Now, $w^+\neq v_{d+2}$ since $P'$ is chordless, implying $dist_G(v_0,w^+)\leq d-1$.
On the other hand, $dist_G(v_0,w^+)\geq d-1$;
otherwise, $v_0Qw^+v_d$ is a $(v_0,v_d)$-path of length at most $d-1$,
contradicting the fact that $P$ is a diameter path.
Hence, $dist_G(v_0,w^+)=d-1$, implying that $dist_G(v_0,w)=d-2$ and $w^+v_{d+2}\in E(Q)$.
Since $v_{d+2},v_{d+3}\in R$, we have $G[\{v_{d+3},v_{d+2},v_d,w^+,w,w^-\}]\simeq S_{1,2,2}$,
a contradiction. Hence, $\ell\leq d+2$.
\end{proof}

{\bf Claim 3:} $C\cup N_G(C)=V(G)$, and every vertex in $V(G)\backslash V(C)$
has at least 2 neighbors in $C$.

\begin{proof}
Suppose that a vertex $x\in V(G)\backslash V(C)$
at distance 1 from $C$ has exactly one neighbor in $C$,
and set $N_C(x)=\{y\}$. And let $z_1, z_2\in N_C^2(x)$, and
let $z_1', z_2'\in N_C^3(x)$.
Then, we have $G[\{x,y,z_1, z_2,z_1', z_2'\}]\simeq S_{1,2,2}$,
a contradiction.

Secondly, suppose, to the contrary, that $z\in V(G)$
is at distance 2 from $C$, and $y$ is a neighbor of $z$
at distance 1 from $C$. Then, $dist_G(z,P)\geq 2$;
otherwise, $y=v_0$ or $v_d$, without loss of generality,
we assume $y=v_0$. Then, $v_1$ must be adjacent to $v_{d+l-1}$,
and thus, $G[\{z,y,v_1,v_2,v_{d+l-1},v_{d+l-2}\}]\simeq N$,
a contradiction.
Hence, $z\in R$. If $y\in R$, then $y$ is not adjacent to any of $v_1,v_2$
and $v_3$. If $y\notin R$, then we have $y\in X$.
Without loss of generality, we assume $y\in X_2$.
Then, $y$ is not adjacent to any of $v_1,v_2$ and $v_3$.
Moreover, from above we know that $y$ has at least 2 neighbors in $C$.
Let $x_1, x_2\in N_C(y)$ be the vertices closest to $v_1$
and $v_3$, respectively. And, let $x_1'$ and $x_2'$
be their neighbors that are closer to $v_1$
and $v_3$ in $C$, respectively.
Then, $G[\{y,z,x_1,x_2,x_1', x_2'\}]\simeq S_{1,2,2}$
if $x_1x_2\notin E(G)$, or
$G[\{y,z,x_1,x_2,x_1', x_2'\}]\simeq N$
if $x_1x_2\in E(G)$, respectively.
Thus, $C$ is a dominating set of $G$.
\end{proof}

By Claims 1 and 2, we know that $C$
is a chordless cycle of length at most $d+l\leq 2d+2$.
Now, we demonstrate a rainbow vertex-coloring of $G$
that uses at most $d+1$ colors.
Relabel $C=x_1x_2\dots x_kx_{k+1}(=x_1),\ 8\leq
2d-2 \leq k\leq 2d+2$.
Then, we assign color $i$ to the vertex $x_i$
if $1\leq i \leq \lceil\frac{k}{2}\rceil$
and assign color $i-\lceil\frac{k}{2}\rceil$
to the vertex $x_i$ if $\lceil\frac{k}{2}\rceil< i \leq k$.
And, we color the remaining vertices arbitrarily.
We can show that this vertex-coloring can make $G$ rainbow vertex-connected.
From Theorem \ref{thm1} and Claim 3, we know that under this vertex-coloring,
pairs in $C\times V(G)$ are rainbow vertex-connected.
And, for each vertex $z\in N_G(C)$, we strength the result of Claim 3 that
$z$ has at least two neighbors colored differently
in $C$. Otherwise, we suppose that $z_1$ and $z_2$
are the only two neighbors of $z$ having the same color
in $C$. From the vertex-coloring, we know that
$dist_C(z_1,z_2)=\lfloor\frac{k}{2}\rfloor \geq 4$.
Then, we can easily find an induced $S_{1,2,2}$, a contradiction.
So, for a pair $(x,y) \in N_G(C)\times N_G(C)$,
we can find a vertex $x'\in N_C(x)$
and a vertex $y'\in N_C(y)$ such that
$x'$ and $y'$ are colored differently.
Since there exists a vertex-rainbow path $P$ connecting
$x'$ and $y'$ and the internal vertices
of $P$ are colored differently from $x'$ and $y'$,
the path $xx'Py'y$ vertex-rainbow connects $x$ and $y$.
Hence, $rvc(G)\leq d+1$.

Up to now, the proof of Theorem 8 is complete.
\end{proof}

Combining Proposition 2 with Theorems 7 and 8, we get
Theorem 6.



\begin{thebibliography}{111}

\bibitem{Bondy} J.A. Bondy, U.S.R. Murty, \emph{Graph Theory with Applications},
The Macmilan Press, London and Basingstoker, 1976.

\bibitem{G.} G. Bacs\'{o}, Zs. Tuza,
\emph{Dominating cliques in $P_5$-free graphs}, Periodica Mathematica Hungarica 21(1990), 303-308.

\bibitem{X. Li4} X. Li, S. Liu, \emph{Tight upper bound of the rainbow vertex-connection number
for 2-connected graphs}, Discrete Appl. Math. 173(2014), 62-69

\bibitem{Char1} G. Chartrand, G.L. Johns, K.A. McKeon, P. Zhang,
\emph{Rainbow connection in graphs}, Math. Bohem. 133(1)(2008), 85-98.

\bibitem{P.1} P. Holub, Z. Ryj\'{a}\v{c}ek, I. Schiermeyer, P. Vr\'{a}na,
\emph{Rainbow connection and foridden subgraphs}, Discrete Math. 338(10)(2015), 1706-1713.

\bibitem{P.2} P. Holub, Z. Ryj\'{a}\v{c}ek, I. Schiermeyer, P. Vr\'{a}na,
\emph{On foridden subgraphs and rainbow connection in graphs with minimum degree 2}, Discrete Math. 338(3)(2015), 1-8.

\bibitem{Chen} L. Chen, X. Li, Y. Shi, \emph{The complexity of determining the rainbow
vertex-connection of graphs}, Theoret. Comput. Sci. 412(2011), 4531-4535.

\bibitem{Yuster} M. Krivelevich, R. Yuster, \emph{The rainbow connection of a graph is (at most) reciprocal to its minimum degree}, J. Graph Theory 63(2010), 185-191.


\bibitem{Li} X. Li, Y. Sun, \emph{Rainbow Connections of Graphs}, New York,
SpringerBriefs in Math., Springer, 2012.

\bibitem{Y. Caro} Y. Caro, A. Lev, Y. Roditty, Z. Tuza,
R. Yuster, \emph{On rainbow connection}, Electron.
J. Comb. 15(2008), R57.

\bibitem{X. Li2} X. Li, Y. Shi, \emph{On the rainbow vertex-connection},
Discuss. Math. Graph Theory 29(2013), 1471-1475.

\bibitem{X. Li3} S. Li, X. Li, Y. Shi, \emph{Note on the complexity of determining the rainbow
(vertex-) connectedness for bipartite graphs},
Appl. Math. Comput. 258(2015), 155-161.

\end{thebibliography}
\end{document}